\documentclass[12pt]{amsart}
\usepackage{amsfonts}
\usepackage{amssymb,latexsym}
\usepackage{enumerate}
\usepackage{mathrsfs}
\makeatletter
\@namedef{subjclassname@2010}{%
  \textup{2010} Mathematics Subject Classification}
\makeatother

\newtheorem{thm}{Theorem}[section]
\newtheorem{cor}[thm]{Corollary}
\newtheorem{lem}[thm]{Lemma}

\newtheorem{prop}[thm]{Proposition}

\theoremstyle{definition}
\newtheorem{defin}[thm]{Definition}
\newtheorem{rem}[thm]{Remark}
\newtheorem{exa}[thm]{Example}

\numberwithin{equation}{section}

\begin{document}


\baselineskip=17pt



\title[Monomial Dynamical Systems over Finite Fields]{Monomial Dynamical Systems of Dimension One over Finite Fields}

\author[Min Sha]{Min Sha}
\address{Department of Mathematical Sciences\\ Tsinghua University\\
100084 Beijing, China} \email{shamin2010@gmail.com}

\author[Su Hu]{Su Hu}
\address{Department of Mathematical Sciences\\ Tsinghua University\\
100084 Beijing, China} \email{hus04@mails.tsinghua.edu.cn}

\date{}

\begin{abstract}
In this paper we study the monomial dynamical systems of dimension
one over finite fields from the viewpoints of arithmetic and graph
theory. We give formulas for the number of periodic points with
period $r$ and for that of cycles with length $r$ respectively.
Then we define and compute both asymptotic mean values and Dirichlet mean values
for periodic points and cycles.
Finally, we associate the monomial dynamical systems with
function fields to assist the computation of mean values.
\end{abstract}



\maketitle

\section{Introduction}        
Finite dynamical systems are discrete-time dynamical systems with
finite state sets. Well-known examples of finite dynamical systems include cellular automata and
Boolean networks, which have found broad applications
in engineering, computer science, and more recently, computational
biology.

Finite dynamical systems over finite fields are most widely studied based on the many good properties of finite fields.
A monomial system of dimension $n$ over a finite field is defined by a
function $f=(f_{1},f_{2},\cdots,f_{n}):\mathbb{F}_{q}^{n} \to
\mathbb{F}_{q}^{n}$, $n \ge 1$, where $\mathbb{F}_{q}$ is a finite
field with $q$ elements, $\mathbb{F}_{q}^{n}$ is the vector space of
dimension $n$ over $\mathbb{F}_{q}$, and $f_{i}:\mathbb{F}_{q}^{n}
\to \mathbb{F}_{q}$ is a monomial of the form
$a_{i}x_{1}^{a_{i1}}x_{2}^{a_{i2}}\cdots x_{n}^{a_{in}}$.

Monomial systems of dimension $n$ over finite fields have
been studied in [2] and [3] from the viewpoint of graph theory,
especially for Boolean monomial systems. T. Vasiga and J. O. Shallit
have obtained several results about tails and cycles in the orbits of
repeated squaring over finite fields in [11], but some of their
results are based on the Extended Riemann Hypothesis. Then
W.-S. Chou and I. E. Shparlinski extended  their results to repeated
exponentiation with any fixed exponent in [1] without resorting to the Extended
Riemann Hypothesis. A. Khrennikov[6], A. Khrennikov and M. Nilsson [7],
M. Nilsson [8],[9] studied the monomial dynamical systems over local
fields.Wu [12] studied the monomial dynamical systems over finite
fields associated with the rational function field $\mathbb{F_{\it
q}}(T)$.

In this paper we study the monomial dynamical systems of dimension
one over finite fields, with the form $x^{n}$, from the viewpoints of arithmetic and graph
theory. We give formulas for the number of periodic points with
period $r$ and for that of cycles with length $r$ respectively.
 Then we define and compute both asymptotic mean values and Dirichlet mean values for periodic points and cycles.
Finally, we generalize Wu's ideas in [12] to study the monomial
dynamical systems over finite fields associated with any function
fields over finite fields, we also point out one mistake in
[12].

Our paper is organized as follows.

Section 2 studies basic properties of monomial systems from the
viewpoints of arithmetic and graph theory, including the properties
of preperiodic points, the conditions for the existence of periodic
points, the number of periodic points with period $r$, the number of
cycles with length $r$, the maximum period, the total number of
periodic points and cycles, the connectivity of the directed graph
associate to the monomial dynamic systems, and so on.

Section 3 introduces mean values of the number of periodic points and
cycles. We define and compute asymptotic mean value and Dirichlet mean value respectively.
In this case, these two mean values coincide.

Section 4 discusses mean values by viewing a
finite field as a residue class field of a function field.
Unfortunately, we find that there are infinitely many cases in which
the asymptotic mean value of the number of fixed points does not exist. It may be
very hard to compute the asymptotic mean value in general cases.
A possible reason is provided therein. Then we define and compute
the Dirichlet mean values for periodic points and cycles.

Section 5 investigates whether the above results are
applicable to the general case with the form $ax^{n}$. We find that
this question depends on ``$a$".

\section{Properties of monomial dynamical systems}    

Let $q$ be a power of a prime number $p$. Let $\mathbb{F}_{ q}$ be a
finite field with $q$ elements. For any positive integer $n \ge 2$,
we consider the dynamical system $f$: $\mathbb{F}_{q}$ $\to$
$\mathbb{F}_{q}$, where
\begin{equation}
 f(x)=x^{n}.
\end{equation}
Let
 $f^{\circ r}$ be the $r$-th iterate of $f$, i.e. $f^{\circ r}=\underbrace{f \circ f \circ \cdots \circ f}_{r~{\rm times}}$.

For every $x \in \mathbb{F}_{q}$, the $orbit$ of $x$ (or an $orbit$
of $f$) is the set
\begin{center}
$\{y\in \mathbb{F}_{q}: \exists k,m \in \mathbb{N}$, such that
$f^{\circ k}(x)=f^{\circ m}(y)\}$.
\end{center}
Obviously the orbits of $f$ give a partition of $\mathbb{F}_{ q}$.

\begin{defin}
{\rm Let $x_{r}=f^{\circ r}(x_{0})$. If
$x_{r}=x_{0}$ for some positive integer $r$, then $x_{0}$ is said to      
be a {\it periodic point} of $f$. If $r$ is the least natural number
with this property, then we call $r$ the {\it period} of $x_{0}$ and
$x_{0}$ an $r$-{\it periodic point}. A
periodic point of period 1 is called a {\it fixed point} of $f$. If
for some $r$, iterate $f^{\circ r}(x_{0})$ is periodic, we call $x_{0}$ a {\it
preperiodic point} of $f$. }
\end{defin}

\begin{defin}
{\rm Let $r$ be a positive integer. The set
$\gamma =\{x_{0},...,x_{r-1}\}$ of periodic points of period $r$ is      
said to be a {\it cycle} of $f$ if $x_{0}=f(x_{r-1})$ and
$x_{j}=f(x_{j-1})$ for $1 \le j \le r-1$. The {\it length of the
cycle} is the number of elements in $\gamma$. We call a $cycle$ with length $r$ an $r$-$cycle$.}
\end{defin}

 Since $\mathbb{F}_{q}$ is a finite set, each element
of $\mathbb{F}_{q}$ must be a preperiodic point of $f$.

By the definition of orbits and the discussion in the above
paragraph, we know that the orbits of $f$ correspond to the cycles
of $f$. Hence the total number of orbits of $f$ is equal to the
total number of cycles of $f$.

All fixed points of $f$, except $x=0$, are solutions of the equation
$x^{n-1}=1$ in  $\mathbb{F}_{q}$. The number of solutions of
the equation $x^{m}=1$ in $\mathbb{F}_{q}$ is given by the
following Lemma.
\begin{lem}
 The equation $x^{m}=1$ has $\gcd(m,q-1)$                          
solutions in $\mathbb{F}_{q}$.
\end{lem}
\begin{proof}
See [5, Proposition 7.1.2].
\end{proof}

All $r$-periodic points, $r \ge 2$, are solutions of the
equation
 $x^{n^{r}-1}=1$. Let $m_{j}= {\rm gcd}(n^{j}-1,q-1)$, $j \ge 1$.
 Here we rewrite the proof of Proposition 5.2 of Section 5 in [12]
 for the convenience of readers.

\begin{prop}
The system $f(x)=x^{n}$ has an r-cycle $(r \ge 2)$ in               
$\mathbb{F}_{q}$ $($i.e. f has an r-periodic point$)$
if and only if $m_{r}$ does not divide any $m_{j}$, $1 \le j \le
r-1$.
\end{prop}
\begin{proof}
For any  $1 \le j \le r$, let $H_{j}=\{\zeta \in
\mathbb{F}_{q}^{*}:\zeta^{n^{j}-1}=1\}=\langle\zeta_{j}\rangle$. Then $H_{j}$ is a subgroup of $\mathbb{F}_{q}^{*}$ and the
number of elements of $H_{j}$ is $|H_{j}|=m_{j}$.

Suppose $f$ has an $r$-cycle ($r \ge 2$). If there exists $1
\le j \le r-1$ such that $m_{r}|m_{j}$, then $H_{r}$ is a subgroup
of $H_{j}$. Hence all $r$-periodic points  belong to
$H_{j}$. By the definition of $r$-periodic point, this
leads to a contradiction.

Conversely, suppose that $m_{r} \nmid m_{j}$ for all $1 \le j \le
r-1$. Notice that $m_{r}$ is the number of nonzero elements
satisfying $f^{\circ r}(x)=x$. So by the hypothesis, for any $j|r$,
$m_{r}>m_{j}$. Hence $f$ must have an $r$-periodic points.
\end{proof}

Let $\mathcal{P}(r,q)$ be the number of $r$-periodic points of $f$.
Let $\mathcal{C}(r,q)$ be the number of $r$-cycles of $f$.
Each $r$-cycle contains $r$ $r$-periodic
points, then
\begin{center}
$\mathcal{C}(r,q)=\mathcal{P}(r,q)/r$.
\end{center}

By Lemma 2.3, $m_{j}+1$ is the number of solutions of $f^{\circ j}(x)=x$
in $\mathbb{F}_{q}$. Hence we have the following relation
between $m_{j}$, $\mathcal{P}(j,q)$ and $\mathcal{C}(j,q)$,
\begin{equation}
m_{j}+1=\sum\limits_{d|j}
\mathcal{P}(d,q)=\sum\limits_{d|j}d\mathcal{C}(d,q), {\rm ~for}~j
\ge 1.
\end{equation}

Let $\mu$ be {\it M\"obius function}. By M\"obius inversion formula
and (2.2), we obtain the following result.

\begin{prop}
 The number of $r$-periodic points of f                                       
and the number of $r$-cycles of f are given respectively
by
\end{prop}
\begin{equation}
\mathcal{P}(r,q)=\sum\limits_{d|r} \mu(d) ({\rm
gcd}(n^{r/d}-1,q-1)+1),
\end{equation}
\begin{equation}
\mathcal{C}(r,q)=\frac{\mathcal{P}(r,q)}{r}=\frac{1}{r}\sum\limits_{d|r}
\mu (d) ({\rm gcd}(n^{r/d}-1,q-1)+1).
\end{equation}

\begin{rem}
 {\rm If $f$ has no $r$-periodic points, then
$\mathcal{P}(r,q)=0$. For example, if $n=2$ and $q=5$, then
$m_{1}=m_{2}=1$. From Proposition 2.4, we get that there are no
2-periodic points, i.e. $\mathcal{P}(2,5)=0$. From
(2.3) we also get that $\mathcal{P}(2,5)=0$. }
\end{rem}

Let $m \ge 2$ be a natural number. Denote the largest divisor of
$q-1$ which is relatively prime to $m$ by $q^{*}(m)$ . If the unique
factorization of $q-1$ is $p_{1}^{e_{1}}p_{2}^{e_{2}} \dots
p_{k}^{e_{k}}$, then $q^{*}(m)=\prod\limits_{p_{i}\nmid
m}p_{i}^{e_{i}}$.

\begin{prop}
For any $\alpha\in\mathbb{F}_{q}^{*}$, $\alpha$ is a periodic point
if and only if its order {\rm ord($\alpha)$}$|q^{*}(n)$. Moreover,
if $\alpha$ is a periodic point, the length of this cycle equals to
the exponent of $n$ modulo {\rm ord($\alpha)$}, and each element in
this cycle has the same order.
\end{prop}
\begin{proof}
See the first paragraph in the proof of [1, Theorem 1].
\end{proof}

We use $\varphi$ to denote {\it Euler's $\varphi$-function}. We have
the following lemma.

\begin{lem}
We have
 \begin{center}
 {\rm gcd}$(n^{r}-1,q-1)$ = {\rm gcd}$(n^{r}-1,q^{*}(n))$,
 \end{center}for r $\ge 1$.
\end{lem}

Let $\hat{r}(n)$ be the exponent of $n$ modulo $q^{*}(n)$, i.e.
$\hat{r}(n)$ is the least positive integer satisfying
$q^{*}(n)|(n^{\hat{r}(n)}-1)$.

M.Nilsson [8] proved the following result for p-adic dynamical
systems. We follow his idea to get a similar result for dynamical
systems over finite fields.

\begin{prop}
If M is a positive integer not less than $
\hat{r}(n)$, then                                                     
\begin{equation}
\sum\limits_{r=1}^{M}\mathcal{P}(r,q) = q^{*}(n)+1.  \notag
\end{equation}
\end{prop}
\begin{proof}
First we prove that $\mathcal{P}(r,q)=0$ if $r>\hat{r}(n)$. Since
$m_{\hat{r}(n)}=$gcd$(n^{\hat{r}(n)}-1,q-1)$ =
gcd$(n^{\hat{r}(n)}-1,q^{*}(n))$ = $q^{*}(n)$, and by Lemma 2.8, we
have $m_{r}|m_{\hat{r}(n)}$,  for every $r>\hat{r}(n)$. We have
$\mathcal{P}(r,q)=0$ by Proposition 2.4.

Next we prove that if $r \nmid \hat{r}(n)$ then
$\mathcal{P}(r,q)=0$. Let $l_{1}$ be a positive divisor of
$q^{*}(n)$. Let $s$ be the least positive integer such that $n^{s}-1
\equiv 0 ($mod $l_{1})$. By the division algorithm we have
$\hat{r}(n)=ks+r_{1}$, where $k,r_{1} \in \mathbb{Z}$ and $0\le r_{1}<s$. Since
$n^{\hat{r}(n)} \equiv 1 ($mod $q^{*}(n))$,  we have $n^{\hat{r}(n)}
\equiv 1 ($mod $l_{1})$. This implies that
\begin{center}
$1 \equiv n^{\hat{r}(n)} \equiv n^{ks+r_{1}} \equiv
(n^{s})^{k}n^{r_{1}} \equiv n^{r_{1}}($mod $l_{1})$.
\end{center}
By the definition of $s$, we have $r_{1}=0$. Thus $s|\hat{r}(n)$.

Suppose $r \nmid \hat{r}(n)$. Since
$m_{r}=$gcd$(n^{r}-1,q-1)=$gcd$(n^{r}-1,q^{*}(n))|q^{*}(n)$, if $s$
is the least positive number such that $n^{s}-1 \equiv 0 ($mod
$m_{r})$, we have $s|\hat{r}(n)$. By the definition of $s$ we must
have $s<r$. But $m_{r}|m_{s}$. Hence $\mathcal{P}(r,q)=0$ by
Proposition 2.4.

Thus
\begin{center}
$\sum\limits_{r=1}^{M}\mathcal{P}(r,q)=\sum\limits_{r|\hat{r}(n)}\mathcal{P}(r,q)$.
\end{center}

By (2.2) and Lemma 2.8, we have
\begin{center}
 $\sum\limits_{d|r} \mathcal{P}(d,q)=$gcd$(n^{r}-1,q^{*}(n))+1$.
 \end{center}
Finally,we get our result by setting $r=\hat{r}(n)$ in the above
formula.
\end{proof}

\begin{rem}
{\rm From the proof of Proposition 2.9, we can only have $r$-cycles
 (i.e. $r$-periodic points ) such that $r|\hat{r}(n)$.}
\end{rem}

\begin{cor}
The maximum length of cycles of f is $\hat{r}(n)$.                 
\end{cor}
\begin{proof}
By Remark 2.10, it suffices to show that there exists a $\hat{r}(n)$-cycle of $f$.
 If $\hat{r}(n)=1,$ this result is trivial,
since 0 is a fixed point. If $\hat{r}(n)>1,$ suppose there exists an
integer $j$, $1 \le j <\hat{r}(n)$, such that
$m_{\hat{r}(n)}|m_{j}$. Then $m_{j}=m_{\hat{r}(n)}=q^{*}(n)$. By the
definition of $\hat{r}(n)$, we have $j \ge \hat{r}(n)$. This leads
to a contradiction. Hence we have $m_{\hat{r}(n)}\nmid m_{j}$ for
all $j$, $1 \le j <\hat{r}(n)$. Then we get our result from
Proposition 2.4.
\end{proof}

\begin{cor}
f has $q^{*}(n)+1$ periodic points in $\mathbb{F}_{q}.$             
\end{cor}

If each element of $\mathbb{F}_{q}$ is a periodic point of $f$,
then $q^{*}(n)+1=q$, thus gcd$(q-1,n)=1$. In fact $f$ is bijective
if and only if gcd$(q-1,n)=1$. Therefore, each element of
$\mathbb{F}_{q}$ is a periodic point of $f$ if and only if $f$
is bijective.

\begin{cor} The total number of cycles (or orbits) of f is given by       
\begin{equation}
\sum\limits_{r|\hat{r}(n)}
\mathcal{C}(r,q)=\sum\limits_{r|\hat{r}(n)}\frac{1}{r}\sum\limits_{d|r}
\mu (d) ({\rm gcd}(n^{r/d}-1,q-1)+1).  \notag
\end{equation}
\end{cor}
\begin{proof}
By Remark 2.10 and (2.4), we get the result.
\end{proof}

\begin{exa}
{\rm Let us consider this monomial system $f(x)=x^{2}$ over
$\mathbb{F}_{7}$. From Proposition 2.5 we know that $f$ has two
1-periodic points, two 2-periodic points,
two 1-cycles and one 2-cycle. From Corollary
2.11, Corollary 2.12 and Corollary 2.13, we know that the maximum
lengths of cycles of $f$ is $2$, so $f$ has four periodic points and
three cycles respectively. We can also get the above results by
calculating directly.}
\end{exa}

The dynamics of $f$ can be represented by the {\it state space} of
$f$, denoted by $S(f)$, which is a directed graph. The vertex set of
$S(f)$ is $\mathbb{F}_{q}$. We draw a directed edge from $a$ to
$b$ if $f(a)=b$. It should be noted that a directed edge from a
vertex to itself is admissible, and the length of a directed cycle
need not to be greater than $2$. Hence $S(f)$ encodes all state
transitions of $f$, and has the property that every vertex has
out-degree exactly equal to $1$. Moreover, a connected component of
$S(f)$ coincides with an orbit of $f$, and a directed cycle of
$S(f)$ coincides with a cycle of $f$ which has been defined in
Definition 2.2. We have the following proposition on the
connectivity of $S(f)$.

\begin{prop}
$S(f)$ is not connected.
\end{prop}
\begin{proof}
First we claim that each connected component of $S(f)$ contains only
one directed cycle. Since if one connected component contains two or
more cycles, then there exists one vertex of this component which
has out-degree not equal to $1$.

Note that $\{0\}$ and $\{1\}$ are two cycles of $f$, then there are
more than one directed cycles of $S(f)$. Hence $S(f)$ is not
connected.
\end{proof}

From the proof of Proposition 2.15, we know that each connected
component of $S(f)$ contains only one directed cycle. Then the
number of connected components of $S(f)$ equals to the number of
cycles of $f$.

We can get a subdigraph of $S(f)$, denoted by $S(f^{*})$, which is
induced by $f^{*}=f|_{\mathbb{F}_{q}^{*}}$. That is we obtain
$S(f^{*})$  from $S(f)$ by deleting the vertex $\{0\}$. We also have
the following Proposition on the connectivity of $S(f^{*})$.

\begin{prop}

{\rm (1)~}$S(f^{*})$ is connected if and only if $q^{*}(n)=1$.

{\rm (2)~}$S(f^{*})$ is not strongly connected.
\end{prop}
\begin{proof}
(1)Suppose that $S(f^{*})$ is connected. Then $S(f^{*})$ has only
one connected component, so $S(f^{*})$ has only one directed cycle,
then $f^{*}$ has only one cycle, that is $\{1\}$. So $f$ has only
two periodic points. By Corollary 2.12, we have $q^{*}(n)=1$.

Conversely, suppose that $q^{*}(n)=1$. By Corollary 2.12, $f^{*}$
has only one cycle. Hence $S(f^{*})$ is connected.

(2)Suppose that $S(f^{*})$ is strongly connected. Then $f^{*}$ has
only one cycle, and each element of $\mathbb{F}_{q}^{*}$ lies in
this cycle. But $f^{*}$ has a 1-cycle, that is $\{1\}$, and no other elements
of $\mathbb{F}_{q}^{*}$ lie in this cycle.
\end{proof}

We call $f$ a {\it fixed point system}(See [2],[3]), if all directed
cycles of $S(f)$ have length $1$, that is the number of vertex in
every strongly connected component of $S(f)$ is one. By Corollary
2.11, we have the following proposition.

\begin{prop}
$f$ is a fixed point system if and only if $\hat{r}(n)=1$, i.e.
$q^{*}(n)|(n-1)$.
\end{prop}

\section{Mean values of periodic points and cycles}        

In this section we discuss the mean values of periodic points and
cycles of $f$ in $\mathbb{F}_{q}$ with respect to $q$.
Let $\tau(m)$ be the number of positive divisors of the positive
integer $m$. Let $t$ be a positive integer. Let $\pi(t)$ be the
number of primes less than or equal to $t$.

\begin{defin}                                                        
    {\rm Let $s$ be a positive integer. If the limit $\lim\limits_{t \to \infty}\frac{1}{\pi(t)}\sum\limits_{{\rm }p \le t}{\mathcal{P}}(r,p^{s})$
   exists (including $\infty$), then we call it the {\it asymptotic mean value} of the number of $r$-periodic points of $f$ in $\mathbb{F_{\it p^{s}}}$, when $p\to \infty$, and denote it by $N(r,s)$.}
\end{defin}

Similarly, we can also define the asymptotic mean values for $r$-cycles, for all periodic points and for all cycles, respectively.

Let $l$ and $s$ be two positive integers respectively. Let
$v_{s}(l)$ be the number of solutions of $x^{s}= 1$ in $\mathbb{Z}/l\mathbb{Z}$. M.Nilsson [9] proved the following theorem.
\begin{thm}
{\rm (M.Nilsson)}

Let t and m be two positive integers, then
   \begin{equation}
      I_{m}(s)\triangleq\lim\limits_{t \to \infty}\frac{1}{\pi(t)}\sum\limits_{{\rm }p \le t}{\rm gcd}(m,p^{s}-1)=\sum\limits_{l|m}v_{s}(l).
   \end{equation}
In particular if s=$1$, then
  \begin{equation}
      I_{m}(1)=\lim\limits_{t \to \infty}\frac{1}{\pi(t)}\sum\limits_{{\rm }p \le t}{\rm gcd}(m,p-1)=\tau(m), \notag
  \end{equation}
where the sum is over all primes p $\le$ t.
\end{thm}
\begin{proof}
 See [9, Theorem 6.2].
\end{proof}

\begin{rem}
{\rm From Lemma 2.3, we know that (3.1) is the asymptotic mean value
of the number of solutions of $x^{m}=1$  in $\mathbb{F_{\it p^{s}}}$, when $p\to \infty$.}
\end{rem}

We have the following proposition on the asymptotic mean value of the number of
$r$-periodic points.

\begin{prop}
If r is a positive integer, then       
\begin{equation}
\begin{array}{lll}
 N(r,s)
     &=&\sum\limits_{d|r}\mu(d)(\sum\limits_{l|(n^{r/d}-1)}v_{s}(l)+1),\\
     &=&\left\{ \begin{array}{ll}
                 \sum\limits_{l|(n-1)}v_{s}(l)+1 & \textrm{if $r=1$}\\
                  \\
                 \sum\limits_{d|r}\sum\limits_{l|(n^{r/d}-1)}\mu(d)v_{s}(l) & \textrm{if $r>1$}
                 \end{array} \right.
\end{array}
\notag
\end{equation}
\text{\it where the sum is over all primes $p \le t$}.
\end{prop}
\begin{proof}
By (2.3) we have
   \begin{center}
     $\mathcal{P}(r,p^{s})=\sum\limits_{d|r} \mu(d) ({\rm gcd}(n^{r/d}-1,p^{s}-1)+1).$
   \end{center}
Notice that if $r>1$, $\sum\limits_{d|r}\mu(d)=0$. Then this
result follows directly from Theorem 3.2.
\end{proof}

\begin{rem}
  {\rm M.Nillson [9] proved that $I_{m}(s)$ is a periodic function in $s$, and he also gave a formula for
   its period.}
\end{rem}

\begin{exa}
{\rm
 If $n=2$, the asymptotic mean value of the number of fixed points is $2$ by Proposition 3.4. This can also be checked directly.
 Since $f$ has only two fixed points $0$ and $1$ in every $\mathbb{F_{\it p^{s}}}$, the asymptotic mean value of the number of
 fixed points is $2$.
}
\end{exa}

We can also get the following proposition about the asymptotic mean value of the number of $r$-cycles.

\begin{prop}
Let r be a positive integer, then                              
  \begin{equation}
\begin{array}{lll}
 \lim\limits_{t \to \infty}\frac{1}{\pi(t)}\sum\limits_{p \le t}{\mathcal{C}}(r,p^{s})
     &=&\frac{1}{r}\sum\limits_{d|r}\mu(d)(\sum\limits_{l|(n^{r/d}-1)}v_{s}(l)+1),\\
     &=&\left\{ \begin{array}{ll}
                 \sum\limits_{l|(n-1)}v_{s}(l)+1 & \textrm{if $r=1$}\\
                  \\
                 \frac{1}{r}\sum\limits_{d|r}\sum\limits_{l|(n^{r/d}-1)}\mu(d)v_{s}(l) & \textrm{if $r>1$}
                 \end{array} \right.
\end{array}
\notag
\end{equation}
\text{\it where the sum is over all primes $p \le t$}.
\end{prop}
\begin{proof}
It follows directly from Proposition 3.4, since
${\mathcal{C}}(r,p^{s})=\frac{1}{r}{\mathcal{P}}(r,p^{s}).$
\end{proof}

The asymptotic mean value of the number of periodic points is equal to
$\sum\limits_{r=1}^{\infty}N(r,s)$. If $r$ is a prime number, then
$N(r,s)=\sum\limits_{l|(n^{r}-1)}v_{s}(l)-\sum\limits_{l|(n-1)}v_{s}(l)$.
 Since $n^{r}-1=(n-1)(n^{r-1}+\cdots+n+1)$ and $n \ge 2,$ then $N(r,s) \ge 1.$ Thus $\sum\limits_{r=1}^{\infty}N(r,s)=\infty$.

The asymptotic mean value of the number of cycles equals to
$\sum\limits_{r=1}^{\infty}\frac{1}{r}N(r,s)$.
 If $r$ is a prime number, then $\frac{1}{r}N(r,s) \ge \frac{1}{r}$. Note that $\sum\limits_{p}\frac{1}{p}$, where the sum is over all
 prime numbers $p$, is equal to $\infty$. Hence the asymptotic mean value of the number of cycles is infinite.

From the discussions in the above two paragraphs we get the
following proposition.

\begin{prop}
 The asymptotic mean value of the number of periodic points and that of the number of cycles are both infinite.      
\end{prop}

In fact, W.-S. Chou and I.E. Shparlinski give a bound for the
asymptotic mean value of the number of periodic points in [1, Theorem 2].

We recall the definition of Dirichlet density for natural numbers.
Let $A$ be a set of primes in $\mathbb{N}$. If the limit
  \begin{center}
   $\lim\limits_{s \to 1^{ +}}\frac{\sum\limits_{p \in A}p^{-s}}{\sum\limits_{p\in \mathbb{N}}p^{-s}}$
  \end{center}
exists, then we call it the {\it Dirichlet density} of $A$, and
denote it by $\delta(A)$.

The following lemma is another form of M$\ddot{{\rm o}}$bius
inversion formula.
\begin{lem}
 Let $m$ be a positive integer and $g(r)=\sum\limits_{kr|m}h(kr)$. Then        
 \begin{center}
   $h(r)=\sum\limits_{kr|m}\mu(k)g(kr)$.
 \end{center}
\end{lem}
\begin{proof}
 See [9, Lemma 6.1].
\end{proof}

We follow M.Nilsson's idea in [9, Theorem 6.2] to get the following
lemma.

\begin{lem}
 Let $m$ be a positive integer, $l$ be a positive divisor of $m$, and $A(l,m)=\{p\in \mathbb{N}:{\rm gcd}(m,p^{s}-1)=l\}$.
 Then
 \begin{center}
  $\sum\limits_{l|m}l\delta(A(l,m))=\sum\limits_{l|m}v_{s}(l)$.        
 \end{center}
\end{lem}
\begin{proof}
Let $B(l,m)=\{p\in \mathbb{N}:l|(p^{s}-1)\}$. Obviously,
$B(l,m)=\bigcup\limits_{kl|m}A(kl,m)$. In fact this is a disjoint
union. Thus $\delta(B(l,m))=\sum\limits_{kl|m}\delta(A(kl,m))$. By
Lemma 3.9, we have
$\delta(A(l,m))=\sum\limits_{kl|m}\mu(k)\delta(B(kl,m))$. On the
other hand, $B(l,m)=\bigcup\limits_{i^{s}-1\equiv 0({\rm
mod}~l)}C(i,l)$, where $i\le l$ and $C(i,l)=\{p\in \mathbb{N}:p\equiv i({\rm
mod}~l)\}$. If $i^{s}-1\equiv 0({\rm mod}~l)$, then gcd$(i,l)=1$. By
Dirichlet's theorem for arithmetic progressions,
 we have $\delta(C(i,l))=\frac{1}{\varphi(l)}$, where $\varphi$ is Euler's $\varphi$-function.
Hence $\delta(B(l,m))=\frac{1}{\varphi(l)}v_{s}(l)$. So
$\delta(A(l,m))=\sum\limits_{kl|m}\mu(k)\frac{1}{\varphi(kl)}v_{s}(kl)$.
Therefore
 \begin{equation}
  \begin{array}{lll}
   \sum\limits_{l|m}l\delta(A(l,m))&=&\sum\limits_{l|m}l\sum\limits_{kl|m}\mu(k)\frac{1}{\varphi(kl)}v_{s}(kl)\\
   &=&\sum\limits_{d|m}\sum\limits_{k|d}\frac{d}{k}\mu(k)\frac{v_{s}(d)}{\varphi(d)}\\
   &=&\sum\limits_{d|m}\frac{v_{s}(d)}{\varphi(d)}\sum\limits_{k|d}\frac{d}{k}\mu(k)\\
   &=&\sum\limits_{d|m}v_{s}(d),
  \end{array}
  \notag
 \end{equation}
 since $\varphi(d)=\sum\limits_{k|d}\frac{d}{k}\mu(k)$.
\end{proof}

\begin{rem}
{\rm Similar to Remark 3.3, we call
$\sum\limits_{l|m}l\delta(A(l,m))$ the {\it Dirichlet mean value}
of the number of solutions of $x^{m}=1$ in $\mathbb{F_{\it p^{s}}}$
with respect to $p$.}
\end{rem}

Formally we compute the sum of the number of $r$-periodic points of $f$ over all prime numbers as follows,
\begin{equation}
 \begin{array}{lll}
  \sum\limits_{p}\mathcal{P}(r,p^{s})&=&\sum\limits_{p}\sum\limits_{d|r} \mu(d) ({\rm gcd}(n^{r/d}-1,p^{s}-1)+1)\\
  &=&\sum\limits_{d|r}\mu(d)\sum\limits_{p}({\rm gcd}(n^{r/d}-1,p^{s}-1)+1).
 \end{array}
 \notag
\end{equation}
Then we define the {\it Dirichlet mean value} of the number of $r$-periodic points of $f$ in
 $\mathbb{F_{\it p^{s}}}$ with respect to $p$ by
 \begin{center}
  $\sum\limits_{d|r}\mu(d)\sum\limits_{l|(n^{r/d}-1)}(l+1)\delta(A(l,n^{r/d}-1))$,
 \end{center}
denoted by $D(r,s)$.

\begin{prop}
 Let r be a positive integer, then                 
 \begin{center}
  $D(r,s)=\left\{ \begin{array}{ll}
                 \sum\limits_{l|(n-1)}v_{s}(l)+1 & \textrm{if $r=1$}\\
                  \\
                 \sum\limits_{d|r}\sum\limits_{l|(n^{r/d}-1)}\mu(d)v_{s}(l) & \textrm{if $r>1$}
                 \end{array} \right.$
 \end{center}
\end{prop}
\begin{proof}
 By the definition of $D(r,s)$ and Lemma 3.10, we have
 \begin{equation}
  \begin{array}{lll}
   D(r,s)&=&\sum\limits_{d|r}\mu(d)\sum\limits_{l|(n^{r/d}-1)}(l+1)\delta(A(l,n^{r/d}-1))\\
   &=&\sum\limits_{d|r}\mu(d)(\sum\limits_{l|(n^{r/d}-1)}l\delta(A(l,n^{r/d}-1))+1)\\
   &=&\sum\limits_{d|r}\mu(d)(\sum\limits_{l|(n^{r/d}-1)}v_{s}(l)+1).
   \end{array}
    \notag
 \end{equation}
\end{proof}

\begin{rem}
{\rm Comparing with Proposition 3.4, we see that the asymptotic mean value and Dirichlet mean value of the number of $r$-periodic points are the same. Hence the Dirichlet mean values of the number of $r$-cycles, of the number of periodic points and of the number of cycles are the
same as those for asymptotic mean values respectively. }
\end{rem}

\section{Mean values associated with function fields}           

In the previous section, we compute mean values with respect to $p$. To
the contrary, in this section we will fix a prime number $p$ to
compute mean values with respect to the powers of $p$.

Let $\mathbb{F}_{q}$ be a finite field  with $q$ elements. Let
$\mathbb{A}=\mathbb{F}_{q}[T]$ be the polynomial ring over
$\mathbb{F}_{q}$. Let $K$ be a function field with constant
field $\mathbb{F}_{q}$, and $\mathcal{O}_{K}$ be the ring of
 integers of $K$ over $\mathbb{A}$. It is well known that
 for every prime ideal $P$ in $\mathcal{O}_{K}$, the
residue class field $\mathcal{O}_{K}/P$ is a finite field, with a
power of $q$ elements, and we call the degree of this extension
$[\mathcal{O}_{K}/P : \mathbb{F}_{q}]$ the $degree$ of $P$,
denoted by deg$P$. We use $|P|$ to denote the number of elements
of $\mathcal{O}_{K}/P $, i.e. $|\mathcal{O}_{K}/P|$.

In this section, we will view a finite field as $\mathcal{O}_{K}$
modulo some $P$.

Let $L/K$ be a finite extension of function fields. It is well known
that if a prime $P$ of $K$ is unramified over $L$, for every prime
$\mathfrak{P}$ of $L$ over $P$, we can associate an automorphism
$(\mathfrak{P},L/K)$ in $G={\rm Gal}(L/K)$ called the {\it Frobenius
automorphism}
 of $\mathfrak{P}$. Moreover, the set of automorphisms
  $(P, L/K) \triangleq \{(\mathfrak{P},L/K)|\mathfrak{P}$ over $P\}$ is a conjugacy class in $G$.

First we recall two results in function fields. One is about the
prime ideals decomposition in constant field extensions. The other
is about the Chebotarev density theorem for the natural density of
primes.

\begin{lem}                  
  Suppose the constant field of a function field K is $\mathbb{F}_{q}$, $K_{n}=\mathbb{F_{\it q^{n}}}K$, $\mathcal{O}$
   and $\mathcal{O}_{n}$ are the rings of integers of K and $K_{n}$ respectively. Then for every prime $P$ of $\mathcal{O}$ with degree
   $d$, we have
   \begin{center}
     $P\mathcal{O}_{n}=\mathfrak{P}_{1}\mathfrak{P}_{2} \cdots \mathfrak{P}_{g}$
   \end{center}
   where each $\mathfrak{P}_{i}$ is a prime of $\mathcal{O}_{n}$, $g={\rm gcd}(n,d)$ and the degree of residue class $f(\mathfrak{\mathfrak{P}}_{i}/P)$,
   i.e. $[\mathcal{O}_{n}/\mathfrak{P}_{i} :\mathcal{O}/P]$, $=\frac{n}{{\rm gcd}(n,d)}$.
\end{lem}
\begin{proof}
See [10, Proposition 8.5 and Proposition 8.13].
\end{proof}

\begin{thm}
{\rm (Chebotarev)}                       

Let L/K be a Galois extension of function fields. Denote {\rm Gal($L/K$)} by
G. Let F and E be the constant fields of K and L respectively, and F
with q elements. Let $[E:F]=l, [L:KE]=m, [K:F(T)]=d$, $g_{K}$ and
$g_{L}$ be the genus of K and L respectively . Let $\mathscr{C}$ be
a conjugacy class in G, $\mathcal{S}_{k}(L/K,\mathscr{C})=\{$ P
unramified prime  of K $: (P,L/K)=\mathscr{C},$ and {\rm
deg}$P=k\}$, and
$\mathcal{C}_{k}(L/K,\mathscr{C})=|\mathcal{S}_{k}(L/K,\mathscr{C})|$.
For every $\rho \in \mathscr{C}$, $\rho|_{E}$ is the same power of
the Frobenius $\sigma$ of E/F, assume $\rho|_{E}=\sigma^{a}, \forall
\rho \in \mathscr{C}$. Then we have

{\rm (1)} $\mathcal{S}_{k}(L/K,\mathscr{C})$ is empty except when $k
\equiv a (${\rm mod }$l)$.

{\rm (2)} If $k \equiv a (${\rm mod }$l)$, then
 \begin{center}
 $|\mathcal{C}_{k}(L/K,\mathscr{C})-\frac{|\mathscr{C}|}{km}q^{k}| < 4|\mathscr{C}|(d^{2}+
\frac{1}{2}g_{L}d+\frac{1}{2}g_{L}+g_{K}+1)q^{k/2}$.
 \end{center}
\end{thm}
\begin{proof}
  See [4, Proposition 5.16].
\end{proof}

In what follows, we fix a prime number $p$ and we also fix a
function field $K$ with constant field $\mathbb{F}_{q}$, where
$q$ is a power of $p$. Let $P_{K}$ be the set of all primes of $K$.
If $t$ is a positive integer, then we denote the set of all primes
$P$ of $K$ with deg$P \le t$ by $P_{K}(t)$, and we also denote
$|P_{K}(t)|$ by $\pi_{K}(t)$.

We recall the definition of density. Let $A$ be a subset of $P_{K}$.
Let $A(t)=\{P\in A: {\rm deg}P \le t\}$. If the limit
$\lim\limits_{t \to \infty}\frac{|A(t)|}{\pi_{K}(t)}$ exists, then
we call it the {\it natural density} of $A$. If the limit
  \begin{center}
   $\lim\limits_{s \to 1^{ +}}\frac{\sum\limits_{P \in A}|P|^{-s}}{\sum\limits_{P \in P_{K}}|P|^{-s}}$
  \end{center}
exists, then we call it the {\it Dirichlet density} of $A$, and
denote it by $\delta(A)$.

We follow Wu's idea ([12, Lemma 4.2]) to get the following lemma.

\begin{lem}
Let $r$ be a positive integer. Let $S(r,P_{K})$ denote the set of all primes in                     
$P_{K}$ such that $r|(|P|-1)$.

{\rm (1)} If ${\rm gcd}(r,q)=1$, then
 \begin{center}
   $\delta(S(r,P_{K}))=\frac{1}{l_{r}}$,
 \end{center}
where $l_{r}$ is the minimal positive integer $l$ such that $q^{l}
\equiv 1 (${\rm mod} $r).$

{\rm (2)} $S(r,P_{K})$ is not empty if and only if ${\rm
gcd}(r,q)=1$.
\end{lem}
\begin{proof}
(1) Assume $|P|=q^{s}$, then $r|(|P|-1)$ if and only if $q^{s}
\equiv 1 (${\rm mod} $r)$. Hence there exists a minimal positive
integer $l_{r}$ such that $q^{l_{r}} \equiv 1 (${\rm mod} $r).$ In
fact $l_{r}$ is the multiplicative order of $q$ modulo $r$. Thus
$q^{s} \equiv 1 (${\rm mod} $r)$ if and only if $l_{r}|s$. Hence
$r|(|P|-1)$ if and only if $l_{r}|s$.

Let $H_{r}=\{\zeta:\zeta^{r}=1\}\subset\overline{\mathbb{F}}_{\it
q}$ (algebraic closure of $\mathbb{F}_{q})$. Let $\zeta_{r}$ be
a generator of $H_{r}$. Let $\mathbb{F}=\mathbb{F_{\it
q}}(\zeta_{r})$. Then $[\mathbb{F}:\mathbb{F}_{q}]=l_{r}$. Let
$K_{r}=K(\zeta_{r})$. Then $K_{r}$ is a constant field extension
over $K$ of degree $l_{r}$. By Lemma 4.1 and the conclusion in the
above paragraph, we know that there is a bijection between
$S(r,P_{K})$ and the set of primes of $\mathcal{O}_K$ which are
split completely in $K_{r}$. By [10, Proposition 9.13], we have
 \begin{center}
    $\delta(S(r,P_{K}))=\frac{1}{l_{r}}$.
 \end{center}

 (2) Suppose that ${\rm gcd}(r,q)=1$. By $(1)$, we know that $S(r,P_{K})$ is not empty. Conversely,
 suppose $S(r,P_{K})$ is not empty. Then there exists $P \in P_{K}$, such that
 $r|(|P|-1)$. Assume $|P|=q^{s}$, then $r|(q^{s}-1)$, thus ${\rm gcd}(r,q)=1$.
\end{proof}

\begin{rem}
{\rm
 Wu ([12, Lemma 4.2]) proved the same result about the natural density of $S(r,P_{K})$ for the case of rational function
fields. But her result is not correct, because she used the
conclusion for the Dirichlet density to compute the natural density.
In fact, if we further assume $r\nmid q-1$, the natural density of
$S(r,P_{K})$ doesn't exist, see Lemma 4.5. }
\end{rem}

Let $S(r,P_{K}(t))$ be the set of all primes in $P_{K}(t)$ such that
$r|(|P|-1)$, and $C(r,P_{K}(t))=|S(r,P_{K}(t)|$. The next lemma
shows that if gcd$(r,q)=1$ and $r\nmid q-1$, then the natural
density of $S(r,P_{K})$ doesn't exist.

\begin{lem}
 If ${\rm gcd}(r,q)=1$, then the limit
  $\lim\limits_{t \to \infty}\frac{C(r,P_{K}(t))}{\pi_{K}(t)}$ equals to 1 or doesn't exist.
  In particular, $\lim\limits_{t \to
  \infty}\frac{C(r,P_{K}(t))}{\pi_{K}(t)}=1$ if and only if $r|q-1$.
\end{lem}
\begin{proof}
We compute the quantities in Theorem 4.2 in our case. It is obvious
that $l=l_{r}$. Since $K_{r}/K$ is a constant field extension, then
$m=1$ and $G={\rm Gal}(K_{r}/K)$ is an abelian group, in fact
$G={\rm Gal}(\mathbb{F}/\mathbb{F}_{q})=<\sigma>$, where $\sigma$ is
the Frobenius of the extension $\mathbb{F}/\mathbb{F}_{q}$. Thus
every $\mathscr{C}$ contains only one element, hence
$|\mathscr{C}|=1$. We use its element to denote each conjugacy
$\mathscr{C}$ of $G$. Thus for every $\sigma^{s} \in G$, $s \in
\{0,1,2,\cdots,l_{r}-1\}$, if
 $\mathcal{S}_{k}(K_{r}/K,\sigma^{s})$ is not empty, then
 $|\mathcal{C}_{k}(K_{r}/K,\sigma^{s})-\frac{1}{k}q^{k}| < cq^{k/2}$, where $c=4(d^{2}+
\frac{1}{2}g_{K_{r}}d+\frac{1}{2}g_{K_{r}}+g_{K}+1)$. Since a prime
ideal $P$ of $\mathcal{O}_{K}$ is split completely in $K_{r}/K$ if
and only if $(P,K_{r}/K)=1$, that is $a=0$. Hence by Theorem 4.2
(1), we have
 \begin{center}
   $S(r,P_{K}(t))=\bigcup\limits_{k \le t, l_{r}|k}\mathcal{S}_{k}(K_{r}/K,1)$.
 \end{center}
So $C(r,P_{K}(t))=\sum\limits_{k \le t,
l_{r}|k}\mathcal{C}_{k}(K_{r}/K,1)$.

Since all prime ideals of $\mathcal{O}_{K}$ are unramified in
$K_{r}$, the set of all primes of $K$ with degree $k$ is
$\bigcup\limits_{s}\mathcal{S}_{k}(K_{r}/K,\sigma^{s})$, where $s$
runs over the set $\{0,1,2,\cdots,l_{r}-1\}$. By Theorem 4.2 (1), we
conclude that $\mathcal{S}_{k}(K_{r}/K,\sigma^{s})$ is not empty if
and only if $k \equiv s (${\rm mod }$l_{r})$. Hence the set of all
primes of $K$ with degree $k$ is
$\mathcal{S}_{k}(K_{r}/K,\sigma^{k})$. Hence
$\pi_{K}(t)=\sum\limits_{k \le
t}\mathcal{C}_{k}(K_{r}/K,\sigma^{k})$.

So we have
 \begin{center}
   $\sum\limits_{k \le t, l_{r}|k}(\frac{q^{k}}{k}-cq^{k/2})<C(r,P_{K}(t))<\sum\limits_{k \le t, l_{r}|k}(\frac{q^{k}}{k}+cq^{k/2})$,

   $\sum\limits_{k \le t}(\frac{q^{k}}{k}-cq^{k/2})<\pi_{K}(t)<\sum\limits_{k \le t}(\frac{q^{k}}{k}+cq^{k/2})$.
 \end{center}

Thus
 \begin{equation}
   \frac{\sum\limits_{k \le t, l_{r}|k}(\frac{q^{k}}{k}-cq^{k/2})}{\sum\limits_{k \le t}(\frac{q^{k}}{k}+cq^{k/2})}
   <\frac{C(r,P_{K}(t))}{\pi_{K}(t)}
   <\frac{\sum\limits_{k \le t, l_{r}|k}(\frac{q^{k}}{k}+cq^{k/2})}{\sum\limits_{k \le t}(\frac{q^{k}}{k}-cq^{k/2})}.
   \notag
 \end{equation}
for large enough $t$.

First we have
\begin{equation}
\lim\limits_{t \to \infty}\frac{\sum\limits_{k \le t, l_{r}|k}(\frac{q^{k}}{k}-cq^{k/2})}{\sum\limits_{k \le t}(\frac{q^{k}}{k}+cq^{k/2})}
=
\lim\limits_{t \to \infty}\frac{\sum\limits_{k \le t, l_{r}|k}(\frac{q^{k}}{k}+cq^{k/2})}{\sum\limits_{k \le t}(\frac{q^{k}}{k}-cq^{k/2})}
=
\lim\limits_{t \to \infty}\frac{\sum\limits_{k \le t, l_{r}|k}\frac{q^{k}}{k}}{\sum\limits_{k \le t}\frac{q^{k}}{k}}.
\notag
\end{equation}

If $l_{r}=1$, that is $r|q-1$, then $\lim\limits_{t \to \infty}\frac{\sum\limits_{k \le t, l_{r}|k}\frac{q^{k}}{k}}{\sum\limits_{k \le t}\frac{q^{k}}{k}}=1$, i.e.
 $\lim\limits_{t \to \infty}\frac{C(r,P_{K}(t))}{\pi_{K}(t)}=1. $

Otherwise if $l_{r}>1$, suppose that  $\lim\limits_{t \to \infty}\frac{C(r,P_{K}(t))}{\pi_{K}(t)}$ exists.
We choose a subsequence
$\{\frac{C(r,P_{K}(tl_{r}))}{\pi_{K}(tl_{r})}\}$ of the sequence
$\{\frac{C(r,P_{K}(t))}{\pi_{K}(t)}\}$. Then we have
 \begin{equation}
   \frac{\sum\limits_{k \le t}(\frac{q^{kl_{r}}}{kl_{r}}-cq^{kl_{r}/2})}{\sum\limits_{k \le tl_{r}}(\frac{q^{k}}{k}+cq^{k/2})}
   <\frac{C(r,P_{K}(tl_{r}))}{\pi_{K}(tl_{r})}
   <\frac{\sum\limits_{k \le t}(\frac{q^{kl_{r}}}{kl_{r}}+cq^{kl_{r}/2})}{\sum\limits_{k \le tl_{r}}(\frac{q^{k}}{k}-cq^{k/2})}.
   \notag
 \end{equation}

Thus
\begin{equation}
\lim\limits_{t \to \infty}\frac{\sum\limits_{k \le
t}(\frac{q^{kl_{r}}}{kl_{r}}+cq^{kl_{r}/2})} {\sum\limits_{k \le
tl_{r}}(\frac{q^{k}}{k}-cq^{k/2})}
=
   \lim\limits_{t \to \infty}\frac{\sum\limits_{k \le t}(\frac{q^{kl_{r}}}{kl_{r}}-cq^{kl_{r}/2})}
{\sum\limits_{k \le tl_{r}}(\frac{q^{k}}{k}+cq^{k/2})}
=
\lim\limits_{t \to \infty}
\frac{\sum\limits_{k \le t}\frac{q^{kl_{r}}}{kl_{r}}}{\sum\limits_{k
\le tl_{r}}\frac{q^{k}}{k}}. \notag
\end{equation}

 Let $b_{k}=\frac{q^{kl_{r}}}{kl_{r}}, a_{k}=\sum\limits_{i=(k-1)l_{r}+1}^{kl_{r}}\frac{q^{i}}{i}$,
then $\lim\limits_{t \to \infty} \frac{\sum\limits_{k \le
t}\frac{q^{kl_{r}}}{kl_{r}}}{\sum\limits_{k \le
tl_{r}}\frac{q^{k}}{k}}= \lim\limits_{t \to
\infty}\frac{\sum\limits_{k \le t}b_{k}}{\sum\limits_{k \le
t}a_{k}}$.

By Stolz Theorem, we have
\begin{equation}
\lim\limits_{t \to \infty}\frac{\sum\limits_{k \le
t}b_{k}}{\sum\limits_{k \le t}a_{k}}= \lim\limits_{k \to
\infty}\frac{b_{k}}{a_{k}}
=\frac{1}{\frac{1}{q^{l_{r}-1}}+\frac{1}{q^{l_{r}-2}}+\cdots+\frac{1}{q}+1}.
\notag
\end{equation}

Hence
 \begin{equation}
   \lim\limits_{t \to \infty}\frac{C(r,P_{K}(tl_{r}))}{\pi_{K}(tl_{r})}
   =\frac{1}{\frac{1}{q^{l_{r}-1}}+\frac{1}{q^{l_{r}-2}}+\cdots+\frac{1}{q}+1}.
 \end{equation}

We choose another subsequence
$\{\frac{C(r,P_{K}((t+1)l_{r}-1))}{\pi_{K}((t+1)l_{r}-1)}\}$ of the
sequence $\{\frac{C(r,P_{K}(t))}{\pi_{K}(t)}\}$. Then we have

\begin{equation}
   \frac{\sum\limits_{k \le t}(\frac{q^{kl_{r}}}{kl_{r}}-cq^{kl_{r}/2})}{\sum\limits_{k \le (t+1)l_{r}-1}(\frac{q^{k}}{k}+cq^{k/2})}
   <\frac{C(r,P_{K}((t+1)l_{r}-1))}{\pi_{K}((t+1)l_{r}-1)}
   <\frac{\sum\limits_{k \le t}(\frac{q^{kl_{r}}}{kl_{r}}+cq^{kl_{r}/2})}{\sum\limits_{k \le (t+1)l_{r}-1}(\frac{q^{k}}{k}-cq^{k/2})}.
   \notag
 \end{equation}

Thus
\begin{equation}
\lim\limits_{t \to \infty}\frac{\sum\limits_{k \le
t}(\frac{q^{kl_{r}}}{kl_{r}}+cq^{kl_{r}/2})} {\sum\limits_{k \le
(t+1)l_{r}-1}(\frac{q^{k}}{k}-cq^{k/2})}
=
 \lim\limits_{t \to \infty}\frac{\sum\limits_{k \le t}(\frac{q^{kl_{r}}}{kl_{r}}-cq^{kl_{r}/2})}
 {\sum\limits_{k \le (t+1)l_{r}-1}(\frac{q^{k}}{k}+cq^{k/2})}
=
\lim\limits_{t \to \infty}
\frac{\sum\limits_{k \le t}\frac{q^{kl_{r}}}{kl_{r}}}{\sum\limits_{k
\le (t+1)l_{r}-1}\frac{q^{k}}{k}}. \notag
\end{equation}

Let $b_{k}=q^{kl_{r}}/(kl_{r})$ and $a_{1}=\sum_{i=1}^{2l_{r}-1}q^{i}/i$, while
$a_{k}=\sum_{i=kl_{r}}^{(k+1)l_{r}-1}q^{i}/i$ if $k>1$. then $\lim\limits_{t \to \infty} \frac{\sum\limits_{k \le
t}\frac{q^{kl_{r}}}{kl_{r}}}{\sum\limits_{k \le
(t+1)l_{r}-1}\frac{q^{k}}{k}}= \lim\limits_{t \to
\infty}\frac{\sum\limits_{k \le t}b_{k}}{\sum\limits_{k \le
t}a_{k}}$.
By the Stolz Theorem, we have
\begin{equation}
\lim\limits_{t \to \infty}\frac{\sum\limits_{k \le
t}b_{k}}{\sum\limits_{k \le t}a_{k}}= \lim\limits_{k \to
\infty}\frac{b_{k}}{a_{k}}
=\frac{1}{1+q+\dots+q^{l_{r}-1}}.
\notag
\end{equation}
Hence
\begin{equation}
   \lim\limits_{t \to \infty}\frac{C(r,P_{K}((t+1)l_{r}-1))}{\pi_{K}((t+1)l_{r}-1)}
   =\frac{1}{1+q+\dots+q^{l_{r}-1}}.
 \end{equation}

But by the hypothesis, we have
\begin{equation}
\lim\limits_{t \to \infty}\frac{C(r,P_{K}(tl_{r}))}{\pi_{K}(tl_{r})}
=\lim\limits_{t \to
\infty}\frac{C(r,P_{K}((t+1)l_{r}-1))}{\pi_{K}((t+1)l_{r}-1)}.
\notag
\end{equation}

Comparing (4.1) with (4.2), we get a contradiction. Hence
$\lim\limits_{t \to \infty}\frac{C(r,P_{K}(t))}{\pi_{K}(t)}$ doesn't
exist if $l_{r}>1$.
\end{proof}

\begin{rem}
{\rm (1) From the proof of Lemma 4.3 (1), we know that if
gcd$(r,q)=1$, then there is a bijection between
 $S(r,P_{K})$ and the set of prime ideals in $\mathcal{O}_{K}$
 which are split completely in
 $K_{r}$. Thus Lemma 4.5 tells us that if gcd$(r,q)=1$ and $r\nmid
 q-1$,
 the natural density of the set of prime ideals of $\mathcal{O}_{K}$
 which are split completely in
 $K_{r}$ doesn't
 exist.

 (2) From the proof of Lemma 4.5, we know that the set of all primes
of $K$ with degree $k$ is $\mathcal{S}_{k}(K_{r}/K,\sigma^{k})$, and
$||\mathcal{S}_{k}(K_{r}/K,\sigma^{k})|-\frac{1}{k}q^{k}| <
cq^{k/2}$.}
\end{rem}

We follow M.Nilsson's idea (see [8, Theorem 4.9]) to get the
following lemma.

\begin{lem}
Suppose $m$ is a positive integer.    
{\rm (1)} Let r be a positive divisor of $m$. Let $\bar{A}(m,r,P_{K})=\{P\in P_{K}:{\rm gcd}(m,|P|-1)=r\}$. Then
  \begin{center}
   $\sum\limits_{r|m}r\delta(\bar{A}(m,r,P_{K}))=\sum\limits_{r|m^{*}}\frac{1}{l_{r}}\varphi(r)$,
  \end{center}
  where $m^{*}$ is the largest divisor of $m$ such that relatively prime to $q$, $\varphi$ is Euler's $\varphi$-function.

  {\rm (2)} If $m$ is a prime number, ${\rm gcd}(m,q)=1$ and $m \nmid q-1$, then  \\
  $\lim\limits_{t \to \infty}\frac{1}{\pi_{K}(t)}\sum\limits_{P \in P_{K}(t)}{\rm gcd}(m,|P|-1)$ doesn't exist.
\end{lem}
\begin{proof}
 (1) By Lemma 4.3(2), if gcd$(r,q)>1$, $S(r,P_{K})$ and $\bar{A}(m,r,P_{K})$ are both empty.
      So we can suppose that ${\rm gcd}(r,q)=1$.
  Note that $S(r,P_{K})=\{P\in P_{K}: r|(|P|-1)\}$,
 we have $S(r,P_{K})=\bigcup\limits_{kr|m^{*}}\bar{A}(m,kr,P_{K})$. Since this is a disjoint union,
we have
$\delta(S(r,P_{K}))=\sum\limits_{kr|m^{*}}\delta(\bar{A}(m,kr,P_{K}))$.
 By Lemma 4.3, we have $\delta(S(r,P_{K}))=\frac{1}{l_{r}}$. Hence we have
 \begin{center}
   $\delta(\bar{A}(m,r,P_{K}))=\sum\limits_{kr|m^{*}}\mu(k)\delta(S(kr,P_{K}))=\sum\limits_{kr|m^{*}}\frac{\mu(k)}{l_{kr}}$.
 \end{center} from Lemma 3.9.
 Therefore
 \begin{equation}
  \begin{array}{lll}
    \sum\limits_{r|m}r\delta(\bar{A}(m,r,P_{K}))&=&\sum\limits_{r|m^{*}}r\delta(\bar{A}(m,r,P_{K}))
    =\sum\limits_{r|m^{*}}r\sum\limits_{kr|m^{*}}\frac{\mu(k)}{l_{kr}}\\
    &=&\sum\limits_{s|m^{*}}\sum\limits_{r|s}r\frac{\mu(\frac{s}{r})}{l_{s}}
    =\sum\limits_{s|m^{*}}\frac{1}{l_{s}}\sum\limits_{r|s}r\mu(\frac{s}{r})\\
    &=&\sum\limits_{s|m^{*}}\frac{1}{l_{s}}\varphi(s),
  \end{array}
  \notag
 \end{equation}
 since $\varphi(s)=\sum\limits_{r|s}r\mu(\frac{s}{r})$.

 (2) Let $t \in \mathbb{N}$, put $B(m,P_{K}(t))=\sum\limits_{P \in P_{K}(t)}{\rm gcd}(m,|P|-1)$.
Let $r$ be a positive divisor of $m$,
  put $\bar{A}(m,r,P_{K}(t))=\{P\in P_{K}(t):{\rm gcd}(m,|P|-1)=r\}$ and $A(m,r,P_{K}(t))=|\bar{A}(m,r,P_{K}(t))|$.
 Thus $B(m,P_{K}(t))=\sum\limits_{r|m}rA(m,r,P_{K}(t))$ and $S(r,P_{K}(t))=\bigcup\limits_{kr|m}\bar{A}(m,kr,P_{K}(t))$.
 So $C(r,P_{K}(t))=\sum\limits_{kr|m}A(m,kr,P_{K}(t))$. We have
 $A(m,r,P_{K}(t))=\sum\limits_{kr|m}\mu(k)C(kr,P_{K}(t))$ by Lemma 3.9.
 Hence
 \begin{equation}
  \begin{array}{lll}
   B(m,P_{K}(t))&=&\sum\limits_{r|m}rA(m,r,P_{K}(t))=\sum\limits_{r|m}r\sum\limits_{kr|m}\mu(k)C(kr,P_{K}(t))\\
   &=&\sum\limits_{s|m}\sum\limits_{r|s}r\mu(\frac{s}{r})C(s,P_{K}(t))=\sum\limits_{s|m}C(s,P_{K}(t))\sum\limits_{r|s}r\mu(\frac{s}{r})\\
   &=&\sum\limits_{s|m}C(s,P_{K}(t))\varphi(s),
  \end{array}
  \notag
 \end{equation}
 since $\varphi(s)=\sum\limits_{r|s}r\mu(\frac{s}{r})$.

 By the hypothesis of this Lemma, we have
 \begin{center}
  $B(m,P_{K}(t))=C(1,P_{K}(t))+C(m,P_{K}(t))\varphi(m)=\pi_{K}(t)+C(m,P_{K}(t))\varphi(m)$.
 \end{center}
 So
 \begin{equation}
  \begin{array}{lll}
  \lim\limits_{t \to \infty}\frac{1}{\pi_{K}(t)}\sum\limits_{P \in P_{K}(t)}{\rm gcd}(m,|P|-1)
  &=&\lim\limits_{t \to \infty}\frac{1}{\pi_{K}(t)}B(m,P_{K}(t))\\
  &=&\lim\limits_{t \to \infty}(1+\frac{1}{\pi_{K}(t)}C(m,P_{K}(t))),
  \end{array}
  \notag
 \end{equation}
by Lemma 4.5, this limit doesn't exist.
\end{proof}

\begin{rem}
{\rm For every prime ideal $P$ in $\mathcal{O}_{K}$, the ring of
integers of $K$, $\mathcal{O}_{K}/P$ is a finite field. Lemma 2.3
tells us that the equation $x^{m}=1 $ has ${\rm gcd}(m,|P|-1)$
solutions in $\mathcal{O}_{K}/P$. Thus we call
$\sum\limits_{r|m}r\delta(\bar{A}(m,r,P_{K}))$ the {\it Dirichlet mean value} of the number of solutions of $x^{m}=1$ in $\mathcal{O}_{K}/P$
with respect to $P$. }
\end{rem}

From the proof of Lemma 4.7 (2), we have
\begin{center}
 $B(m,P_{K}(t))=\sum\limits_{r|m}C(r,P_{K}(t))\varphi(r)=\sum\limits_{r|m^{*}}C(r,P_{K}(t))\varphi(r)$,
\end{center}
for every positive integer $m$.

 Then
\begin{equation}
\lim\limits_{t \to \infty}\frac{1}{\pi_{K}(t)}\sum\limits_{P \in
P_{K}(t)}{\rm gcd}(m,|P|-1)= \lim\limits_{t \to
\infty}\frac{1}{\pi_{K}(t)}\sum\limits_{r|m^{*}}C(r,P_{K}(t))\varphi(r).
\end{equation}
Lemma 4.5 tells us that $\lim\limits_{t \to
\infty}\frac{1}{\pi_{K}(t)}C(r,P_{K}(t))\varphi(r)$ doesn't exist
for each $r|m^{*}$ and $r\nmid q-1$. But it is not easy to answer whether the limit
(4.3) exists in general case. Lemma 4.7 (2) tells us that if
$m^{*}$ is a prime and $m^{*}\nmid q-1$, then the
asymptotic mean value of the number of solutions of $x^{m}=1$ in $\mathcal{O}_{K}/P$
with respect to $P$ doesn't exist.

For every prime ideal $P$ of $\mathcal{O}_{K}$, since
$\mathcal{O}_{K}/P$ is a finite field, we can consider our monomial
system $f:\mathcal{O}_{K}/P \to \mathcal{O}_{K}/P$, where
\begin{equation}
  f(x)=x^{n}, n \ge 2.
\end{equation}

Let $\mathcal{P}(r,P)$ denote the number of $r$-periodic points of $f$.
Let $\mathcal{C}(r,P)$ denote the number of
$r$-cycles of $f$. Then $f$ has all the properties
stated in Section 2. So we can use the conclusions and notations in
Section 2.

We have the following result on the asymptotic mean value of the number of fixed points
in $\mathcal{O}_{K}/P$ with respect to $P$.

\begin{prop}                                                           
 If $n-1$ is a prime number, ${\rm gcd}(n-1,q)=1$ and $(n-1)\nmid q-1$, then
 $\lim\limits_{t \to \infty}\frac{1}{\pi_{K}(t)}\sum\limits_{P \in P_{K}(t)}\mathcal{P}(1,P)$
 doesn't exist. Thus the asymptotic mean value of the number of fixed points of f
 doesn't exist.
\end{prop}
\begin{proof}
 By Proposition 2.5, we have
 \begin{equation}
  \begin{array}{lll}
   &&\lim\limits_{t \to \infty}\frac{1}{\pi_{K}(t)}\sum\limits_{P \in P_{K}(t)}\mathcal{P}(1,P)\\
   &=&\lim\limits_{t \to \infty}\frac{1}{\pi_{K}(t)}\sum\limits_{P \in P_{K}(t)}({\rm gcd}(n-1,|P|-1)+1)\\
   &=&\lim\limits_{t \to \infty}\frac{1}{\pi_{K}(t)}\sum\limits_{P \in P_{K}(t)}{\rm gcd}(n-1,|P|-1)+1.
  \end{array}
  \notag
 \end{equation}
 By the hypothesis of this Proposition and Lemma 4.7 (2), we get the desired result.
\end{proof}

  From Dirichlet's theorem for arithmetic progressions, we know that there are infinitely many
 prime numbers which have the form $n-1$. So there are infinitely many $n$ satisfies the hypothesis of Proposition 4.9.
 In fact it is not easy to compute the asymptotic mean value for periodic points and cycles
 in general case, the reason is the same as that we have discussed below Remark 4.8.
 But we can define and compute the Dirichlet mean value.

Let $r$ be a positive integer, by Proposition 2.5, we have
$\mathcal{P}(r,P)=\sum\limits_{d|r} \mu(d) ({\rm
gcd}(n^{r/d}-1,|P|-1)+1)$. We compute the sum of the number of $r$-periodic points
 of $f$ as follows,
\begin{equation}
 \begin{array}{lll}
  \sum\limits_{P\in P_{K}}\mathcal{P}(r,P)&=&\sum\limits_{P\in P_{K}}\sum\limits_{d|r}\mu(d)({\rm gcd}(n^{r/d}-1,|P|-1)+1)\\
  &=&\sum\limits_{d|r}\mu(d)\sum\limits_{P\in P_{K}}({\rm gcd}(n^{r/d}-1,|P|-1)+1).
 \end{array}
 \notag
\end{equation}
For each $d|r$, recall that $\bar{A}(n^{r/d}-1,k,P_{K})=\{P\in P_{K}:{\rm
gcd}(n^{r/d}-1,|P|-1)=k\}$, then we define the {\it Dirichlet mean value} of the number of $r$-periodic points of $f$ in
   $\mathcal{O}_{K}/P$ with respect to $P$ by
   \begin{center}
     $\sum\limits_{d|r}\mu(d)\sum\limits_{k|(n^{r/d}-1)}(k+1)\delta(\bar{A}(n^{r/d}-1,k,P_{K}))$,
   \end{center}
denoted by $D(r,K)$.

\begin{prop}
Let r be a positive integer. Then                         
 \begin{equation}
  \begin{array}{lll}
  D(r,K)&=&\sum\limits_{d|r}\mu(d)(\sum\limits_{k|(n^{r/d}-1)^{*}}\frac{1}{l_{k}}\varphi(k)+1)\\
  &=&\left\{ \begin{array}{ll}
                 \sum\limits_{k|(n-1)^{*}}\frac{1}{l_{k}}\varphi(k)+1 & \textrm{if $r=1$}\\
                                                     \\
                 \sum\limits_{d|r}\sum\limits_{k|(n^{r/d}-1)^{*}}\frac{1}{l_{k}}\varphi(k)\mu(d) & \textrm{if $r>1$}
              \end{array} \right.
  \end{array}
  \notag
 \end{equation}
where $(n^{r/d}-1)^{*}$ is the largest divisor of $n^{r/d}-1$ such
that relatively prime to $q$.
\end{prop}
\begin{proof}
By Lemma 4.7 (1), we have
 \begin{equation}
  \begin{array}{lll}
    D(r,K)&=&\sum\limits_{d|r}\mu(d)\sum\limits_{k|(n^{r/d}-1)}(k+1)\delta(\bar{A}(n^{r/d}-1,k,P_{K}))\\
   &=&\sum\limits_{d|r}\mu(d)(\sum\limits_{k|(n^{r/d}-1)}k\delta(\bar{A}(n^{r/d}-1,k,P_{K}))+1)\\
   &=&\sum\limits_{d|r}\mu(d)(\sum\limits_{k|(n^{r/d}-1)^{*}}\frac{1}{l_{k}}\varphi(k)+1).
 \end{array}
 \notag
 \end{equation}
 Since if $r>1$, then $\sum\limits_{d|r}\mu(d)=0$, we obtain the desired result.
\end{proof}

\begin{exa}
{\rm
 If $n=2$, then by Proposition 4.10 the Dirichlet mean value of the number of fixed points is $2$. This can also be checked directly.
 Since $f$ has only two fixed points $0$ and $1$ in every $\mathcal{O}_{K}/P$,  the Dirichlet mean value of the number of
 fixed points is $2$.
}
\end{exa}

We call $\frac{D(r,K)}{r}$ the {\it Dirichlet mean value} of the number of
$r$-cycles of $f$ in $\mathcal{O}_{K}/P$ with respect to $P$, denoted by $C(r,K)$. From Proposition 4.10, we
have the following result.

\begin{cor}
If r be a positive integer, then                                        
 \begin{center}
   $C(r,K)=\left\{ \begin{array}{ll}
                 \sum\limits_{k|(n-1)^{*}}\frac{1}{l_{k}}\varphi(k)+1 & \textrm{if $r=1$}\\
                                                     \\
                 \frac{1}{r}\sum\limits_{d|r}\sum\limits_{k|(n^{r/d}-1)^{*}}\frac{1}{l_{k}}\varphi(k)\mu(d) & \textrm{if $r>1$}
              \end{array} \right.$
 \end{center}
 where $(n^{r/d}-1)^{*}$ is the largest divisor of $n^{r/d}-1$ such that relatively prime to $q$.
\end{cor}

We call $\sum\limits_{r \ge 1}D(r,K)$ the {\it Dirichlet mean value} of the number of periodic points.
Note that it may be infinite.

\begin{prop}                                    
 The {\it Dirichlet mean value} of the number of periodic points is infinite.
\end{prop}
\begin{proof}
If $r$ is a prime number, then
$D(r,K)=\sum\limits_{k|(n^{r}-1)^{*}}\frac{1}{l_{k}}\varphi(k)-\sum\limits_{k|(n-1)^{*}}\frac{1}{l_{k}}\varphi(k)$.
Since $n^{r}-1=(n-1)(1+n+\cdots+n^{r-1})$, then
$(n^{r}-1)^{*}=(n-1)^{*}(1+n+\cdots+n^{r-1})^{*}$. So if we can find
infinitely many prime numbers $r$ such that
$(1+n+\cdots+n^{r-1})^{*}>1$, note that for all $k\ge 1$,
$\varphi(k)\ge l_{k}$, then $D(r,K)\ge 1$, then we finish our
proof.

Suppose there are finitely many prime numbers $r$ such that
$(1+n+\cdots+n^{r-1})^{*}>1$, say
$\{r_{1},r_{2},\cdots,r_{m}\}$, and $r_{1}<r_{2}<\cdots<r_{m}$. If a prime
$r>r_{m}$, then $(1+n+\cdots+n^{r-1})^{*}=1$. Hence there exists a
positive integer $s$ such that $1+n+\cdots+n^{r-1}=p^{s}$. Then we have $p\nmid n$. Since there exists a prime
between $r+1$ and $2(r+1)$, say $k$. Then we have $(1+n+\cdots+n^{k-1})^{*}=1$. So
there exists a positive integer $t>s$ such that $1+n+\cdots+n^{k-1}=p^{t}$.
Since $p^{t}=1+n+\cdots+n^{k-1}=(1+n+\cdots+n^{r-1})+n^{r}(1+n+\cdots+n^{k-r-1})=p^{s}+n^{r}(1+n+\cdots+n^{k-r-1})$,
then $p^{s}|(1+n+\cdots+n^{k-r-1})$. If $k=2r+1$, then we have $p^{s}|n^{r}$, a contradiction. Hence we must have $k<2r$.
Then we have $1+n+\cdots+n^{k-r-1}<1+n+\cdots+n^{r-1}=p^{s}$. This leads to a contradiction.
\end{proof}

We call the infinite sum $\sum\limits_{r \ge 1}C(r,K)$ the
 {\it Dirichlet mean value} of the number of cycles.

\begin{prop}                                    
 The {\it Dirichlet mean value} of the number of cycles is infinite.
\end{prop}
\begin{proof}
We arrange the prime numbers by their sizes as
$p_{1}<p_{2}<\cdots<p_{m}<\cdots$, i.e. $p_{1}=2$, $p_{2}=3$ and so
on. Since $\sum\limits_{r \ge 1}C(r,K)\ge \sum\limits_{i
\ge 1}C(p_{i},K) $, if we can show $\sum\limits_{i \ge
1}C(p_{i},K)$ is infinite, our proof is finished.

Since $\sum\limits_{i \ge 1}C(p_{i},K)=\sum\limits_{i \ge
1}\frac{1}{p_{i}}D(p_{i},K)$, note that $\sum\limits_{i \ge
1}\frac{1}{p_{i}}$ is infinite, from the proof of Proposition 4.13,
if there are finitely many prime numbers $r$ such that
$(1+n+\cdots+n^{r-1})^{*}=1$, i.e. there are finitely many prime
numbers $r$ such that $D(p_{i},K)=0$, then we finish our proof.

Suppose there are infinitely many prime numbers $r$ such that
$(1+n+\cdots+n^{r-1})^{*}=1$. Assume $r_{1}$ and $r_{2}$ are two of
them, and $r_{1}<r_{2}$. Then there exist two positive integer
$s_{1}$ and $s_{2}$ such that $1+n+\cdots+n^{r_{1}-1}=p^{s_{1}}$ and
$1+n+\cdots+n^{r_{2}-1}=p^{s_{2}}$. Since
$1+n+\cdots+n^{r_{2}-1}=(1+n+\cdots+n^{r_{1}-1})
+n^{r_{1}}(1+n+\cdots+n^{r_{2}-r_{1}-1})=p^{s_{2}}$ and
$s_{1}<s_{2}$, then
$p^{s_{1}}|n^{r_{1}}(1+n+\cdots+n^{r_{2}-r_{1}-1})$. Suppose that
$r_{2}<2r_{1}$. Then $1+n+\cdots+n^{r_{2}-r_{1}-1}<p^{s_{1}}$, so
$p|n^{r_{1}}$ and $p|n$. But $1+n+\cdots+n^{r_{1}-1}=p^{s_{1}}$, so
$p|1$. This leads to a contradiction. Hence we have $r_{2}\ge
2r_{1}$. From the proof of Proposition 4.13, we conclude that there
are infinitely many prime numbers $r$ such that
$(1+n+\cdots+n^{r-1})^{*}>1$, so we let $p_{k}$ be the minimal prime
number such that $(1+n+\cdots+n^{p_{k}-1})^{*}>1$ and
$(1+n+\cdots+n^{p_{k+1}-1})^{*}=1$. Let $A$ be the set of all the prime
numbers $r$ such that $(1+n+\cdots+n^{r-1})^{*}=1$ and $r\ge
p_{k+1}$. Let $B$ be the set of all the prime numbers $r$ such that
$(1+n+\cdots+n^{r-1})^{*}>1$ and $r\ge p_{k}$. Note that for every
positive integer $m$, there exist a prime number $r$ between $m$ and
$2m$, and for any two adjacent elements $r_{1}, r_{2}$ of $A$, we
have $r_{2}\ge 2r_{1}$. So if we let the elements of $A$ correspond
to the elements of $B$ by their sizes, then we get
$\sum\limits_{r\in A}\frac{1}{r}\le \sum\limits_{r\in
B}\frac{1}{r}$. Since $A\cup B=\{p_{i}:i\ge k\}$ and
$\sum\limits_{i\ge k}\frac{1}{p_{i}}$ is infinite, then
$\sum\limits_{r\in B}\frac{1}{r}$ is infinite. Note that
$\sum\limits_{i \ge 1}C(p_{i},K)\ge \sum\limits_{r\in
B}\frac{1}{r}$, so $\sum\limits_{i \ge 1}C(p_{i},K)$ is infinite.
We finish our proof.
\end{proof}

\section{The general case}    
As suggested by the referee, in this section we investigate whether the above results are applicable to
the general case $f$: $\mathbb{F}_{q}$ $\to$
$\mathbb{F}_{q}$, where
   \begin{equation}
     f(x)=ax^{n}, n\ge2,a\in \mathbb{F}_{q}^{*}.   \notag
    \end{equation}

\begin{lem}(See [5, Proposition 7.1.2])
Let $m \in \mathbb{N}, \alpha\in\mathbb{F}_{q}^{*}$, then the
 equation $\alpha x^{m}=1$ has solutions in $\mathbb{F}_{q}^{*}$ if and only if $\alpha^{\frac{q-1}{d}}=1$,
where $d=\gcd(m,q-1)$. If there are solutions, then there are exactly $d$ solutions.
\end{lem}

Set $G_{m}=\{\alpha\in\mathbb{F}_{q}^{*}:\alpha x^{m}=1$ has
solutions.$\}$, it is easy to see that $G_{m}$ is a subgroup of
$\mathbb{F}_{q}^{*}$. In fact,
$G_{m}=\{\alpha^{m}:\alpha\in\mathbb{F}_{q}^{*}\}$.

All nonzero $r$-periodic points are solutions of the
equation $a^{1+n+\cdots+n^{r-1}}x^{n^{r}-1}=1$. If $f$ has nonzero fixed points,
 i.e. $ax^{n-1}=1$ has solutions in $\mathbb{F}_{q}^{*}$, then $a^{\frac{q-1}{\gcd(n-1,q-1)}}=1$. So
 $(a^{1+n+\cdots+n^{r-1}})^{\frac{q-1}{\gcd(n^{r}-1,q-1)}}=1$. Hence
 $a^{1+n+\cdots+n^{r-1}}x^{n^{r}-1}=1$ has solutions. So there are
 exactly $\gcd(n^{r}-1,q-1)$ solutions. But if $f$ has $r$-periodic points with $r>1$, then the statement ``$f$ has nonzero fixed points" is not true.
 For example, let $q=5, a=3$ and $n=3$, it is easy to check that
 each element in $\mathbb{F}_{5}^{*}$ is a 2-periodic point of $f$, but $f$ does not have nonzero fixed points.

If $a\in G_{n-1}$, then (2.3) and (2.4) are true here according to the above paragraph.
Thus the monomial system considered in this section has $r$-periodic points if and only if the monomial system considered
 in Section 2 has $r$-periodic points, so Proposition
 2.4 is also true here. But Proposition 2.7 is not true here. For example, let $q=5, a=2$ and
 $n=4$. Furthermore, we have the following proposition.
\begin{prop}
If $a\in G_{n-1}$, i.e. $f$ has nonzero fixed points, then all the
results related to $f$ in the above sections are true here except
Proposition 2.7.
\end{prop}

If $a\notin G_{n-1}$, then $f$ has no nonzero fixed points. Thus
(2.3) and (2.4) are not true here. So all the results related to $f$
in section 3 and 4 are not true here. Notice that the identity of
$\mathbb{F}_{q}^{*}$ has order 1, then  Proposition 2.7 is not true
here. Moreover, all the results related to $f$ in section 2 are not
true except Corollary 2.12 and Proposition 2.15. For example, let
$q=3,a=2$ and $n=3$, recall the notations in section 2, we have
$m_{1}=2, m_{2}=2, q^{*}(n)=2$
 and $\hat{r}(n)=1$, $f$ has one fixed point and two 2-periodic points.
 Hence $S(f^{*})$ is strongly connected.
 \begin{lem}
 $f$ has $q^{*}(n)+1$ periodic points in $\mathbb{F}_{q}$.
 \end{lem}
\begin{proof}
Let $k$  be the least common multiple of all periods of periodic points and $\hat{r}(n)$.
Then each periodic point of $f$ satisfies $f^{\circ r}(x)=x$. Conversely, each solution of $f^{\circ r}(x)=x$ is a periodic point of $f$.
Since $f^{\circ r}(x)=x$ has solutions in $\mathbb{F}_{q}^{*}$,
by Lemma 5.1, there are exactly $d=\gcd(n^{r}-1,q-1)$ solutions in $\mathbb{F}_{q}^{*}$.
Since $d|q-1$ and $\gcd(d,n)=1$, we have $d|q^{*}(n)$. Since $q^{*}(n)|q-1$ and $q^{*}(n)|n^{r}-1$, we have $q^{*}(n)|d$.
Hence we get the desired result.
\end{proof}

\begin{prop}
 If $a\notin G_{n-1}$, i.e. $f$ has no nonzero fixed points, then all the results related to $f$ in the above sections
  are not true here except Corollary 2.12 and Proposition 2.15.
\end{prop}

\section{Acknowledgment}
This work was partially supported by NSFC project No.10871107.

We would like to thank Prof. Marcus Nilsson and Prof. Omar
Col$\acute{\rm o}$n-Reyes for sending us their PhD thesis. We would
like to thank Prof. Michael Rosen for answering us our questions on
the density theorem. We also thank Prof.I.E. Shparlinski for sending
us his paper. Special thanks are due to Prof. Linsheng Yin and the
 anonymous referee for  their valuable
suggestions.


\begin{thebibliography}{20}
\bibitem{} W.-S. Chou and I.E. Shparlinski, {\it On the Cycle Structure of
Repeated Exponentiation Modulo a Prime}, Journal of Number Theory,
{\bf 107}(2) (2004), 345-356.

\bibitem{} O.Col$\acute{\rm o}$n-Reyes, R.Laubenbacher, and B.Pareigis, {\it Boolean monomial Dynamical
Systems}, Annals of Combinatorics, {\bf 8}(4) (2004), 425-439.

\bibitem{} O.Col$\acute{\rm o}$n-Reyes, A.S.Jarrah, R.Laubenbacher, and B.Sturmfels, {\it Monomial Dynamical Systems over Finite
Fields}, Complex Systems, {\bf 16}(4) (2006), 333-342.

\bibitem{} M.D.Fried and M.Jarden, {\it Field Arithmetic}, Springer-Verlag, 1986.

\bibitem{} K.Ireland and M.Rosen, {\it A Classical Introduction to Modern Number Theory}, Second Edition, GTM84, Springer-Verlag, 2002.

\bibitem{} A.Khrennikov, {\it Non-Archimedean Analysis: Quantum Paradoxes, Dynamical
Systems and Biological Models}, Kluwer Academic Publishers, 1997.

\bibitem{} A.Khrennikov and M.Nilsson, {\it On the number of cycles of p-adic
dynamical systems}, Journal of Number Theory, {\bf 90}(2) (2001),
255-264.

\bibitem{} M. Nilsson, {\it Cycles of monomials and perturbated monomial p-adic
dynamical systems}, Ann. Math. Blaise Pascal, {\bf 7}(1) (2000),
37-63.

\bibitem{} M.Nilsson, {\it Monomial Dynamics in Finite Extensions of the Field of
p-adic Numbers}, no.05028, School of Mathematics and Systems
Engineering, V$\ddot{\rm a}$xj$\ddot{\rm o}$ University, 2005.

\bibitem{} M.Rosen, {\it Number Theory in Function Fields}, GTM210. Springer-Verlag, 2002.

\bibitem{} T. Vasiga and J.O. Shallit, {\it On the iteration of certain quadratic
maps over GF$(p)$}, Discrete Mathematics, {\bf 277} (2004), 219-240.

\bibitem{} Shuhui Wu, {\it On the average of number of periodic point with period r}, Master's thesis,
  National Central University, Taiwan, 2007.

\end{thebibliography}
\end{document}